\documentclass[a4paper,12pt]{article}
\usepackage[english]{babel}
\usepackage{paralist,url,verbatim, anysize}
\usepackage{amscd}
\usepackage[all,cmtip]{xy}
\usepackage{amsmath,amsthm,amssymb,amsfonts}
\usepackage[bookmarksopen=true]{hyperref}
\usepackage[usenames, dvipsnames]{color}
\usepackage{graphicx}
\usepackage{hyperref}

\makeatletter
\let\@fnsymbol\@arabic
\makeatother

\theoremstyle{plain}
\newtheorem{theorem}{\bf Theorem}[section]
\newtheorem{example1}[theorem]{\bf Example}

\newtheorem{proposition}[theorem]{Proposition}

\newtheorem{problem}[theorem]{Problem}
\newtheorem{question}[theorem]{Question}

\theoremstyle{definition}
\newtheorem{remark}[theorem]{Remark}

\newcommand{\cd}{\operatorname{cd} }
\newcommand{\init}{\operatorname{in} }

\newcommand{\N}{\mathbb{N}}
\newcommand{\Z}{\mathbb{Z}}
\newcommand{\PP}{\mathbb{P}}
\newcommand{\F}{\mathcal{F}}
\newcommand{\D}{\Delta}

\newcommand{\Min}{\operatorname{Min} }

\newcommand{\height}{\operatorname{height} }
\newcommand{\Ker}{\operatorname{Ker} }
\newcommand{\Proj}{\operatorname{Proj} }
\newcommand{\Spec}{\operatorname{Spec} }
\newcommand{\depth}{\operatorname{depth} }

\newcommand{\pp}{\mathfrak{p}}

\newcommand{\kk}{\Bbbk}
\newcommand{\mm}{\mathfrak{m}}
\newcommand{\Lyu}{\mathfrak{L}}
\renewcommand{\O}{\mathcal{O}}

\newcommand{\homo}{\operatorname{hom} }

\definecolor{mypink}{RGB}{215, 5, 234}

\begin{document}
\title{Connectivity of hyperplane sections of domains}
\author{
Matteo Varbaro \thanks{Supported by the \lq\lq program for abroad research activity of the University of Genoa\rq\rq, Pr. N, 100021-2016-MV-ALTROUNIGE$\_$001, and by the GNSAGA-INdAM} \\
\small Dip. di Matematica\\ \small Universit\`{a}  di Genova\\
\footnotesize \url{varbaro@dima.unige.it} }

 \date{}
\maketitle

\begin{flushright} \em \footnotesize
To Gennady Lyubeznik on his 60th birthday \bigskip
\end{flushright}

\begin{abstract}
During the conference held in 2017 in Minneapolis for his 60th birthday, Gennady Lyubeznik proposed the following problem:
Find a complete local domain $R$ and an element $x\in R$ having three minimal primes $\pp_1,\pp_2$ and $\pp_3$ such that $\pp_i+\pp_j$
has height 2 for all $ i\neq j$ and $\pp_1+\pp_2+\pp_3$ has height 4. In this note this beautiful problem will be discussed,
and will be shown that the principle leading to the fact that such a ring cannot exist is false. The specific problem, though, remains open.
\end{abstract}

\section{Introduction}

Given an element $x$ in a domain $R$, at what extent does the ring $R/xR$ remember that $R$ was a domain? 
More precisely, if $R$ is a complete local domain, and $x\in R\setminus \mm$, does the punctured spectrum $\Spec^o (R/xR)$
have any special feature?
The wish to know the answer to this question
led Gennady Lyubeznik to pose the problem mentioned in the abstract. 

The purpose of this note is to convince the reader that $\Spec^o (R/xR)$ does not have special features apart from the obvious ones and from the fact that, if $R$ has dimension at least 3,
it must be connected by the Grothendieck's connectedness theorem \cite[Expos\'{e} XIII, Th\'{e}or\`{e}me 2.1]{SGA2}. What we mean by ``has not special features'' will be made precise in Problem \ref{p:final}: we are not able to solve it, as well as we are not 
able to solve Lyubeznik's problem, however we will exhibit an example of a 4-dimensional $\N$-graded domain $R$ and a linear form $x\in R$ such that 
$\Spec^o (R/xR)$ has an unexpected property (Example \ref{ex:main}), contrary to the prediction of Lyubeznik. On the other hand, we show that his prediction holds in many cases (Theorem \ref{t:pos}).

\section{The Lyubeznik complex of a $\N$-graded $\kk$-algebra}

Let $\Bbbk$ be a separably closed field and $R=\bigoplus_{i\in\N}R_i$ be a $\N$-graded $\kk$-algebra 
(i.e. a $\N$-graded Noetherian ring such that $R_0=\kk$). Denote by $\mm=\bigoplus_{i>0}R_i$.
Notice that the minimal prime ideals of $R$ are homogeneous, so we have
\[\pp\in\Min(R) \implies \pp\subseteq \mm.\]
Therefore, in this situation we can define the simplicial complex $\Lyu(R)$ as follows:
\begin{enumerate}
\item The vertex set of $\Lyu(R)$ is $\{1,\ldots ,s\}=[s]$, where $\Min(R)=\{\pp_1, \ldots, \pp_s\}$.
\item $\{i_1, \ldots, i_k\}$ is a face of $\Lyu(R)$ if and only if $\sqrt{\pp_{i_1} + \ldots + \pp_{i_k}} \neq \mm$.
\end{enumerate}
We will refer to the simplicial complex $\Lyu(R)$ as the \emph{Lyubeznik complex} of $R$. 

\begin{remark}
The Lyubeznik complex can be defined in the same way for any local ring. If $R$ is a $\N$-graded $\kk$-algebra  with homogeneous maximal ideal $\mm$,
then
\[\Lyu(R)=\Lyu(\widehat{R^{\mm}}).\]
(here $\widehat{R^{\mm}}$ denotes the $\mm$-adic completion of $R$). Furthermore, $R$ is a domain if and only if $\widehat{R^{\mm}}$ is a domain: both these facts follow, for example, from \cite[Lemma 1.15]{Va1}. This justifies why we consider $\N$-graded $\kk$-algebras: to provide examples, which is our purpose, that is enough.
\end{remark}

\begin{remark}
In what follows, we will state some results proved in the standard graded case for $\N$-graded $\kk$-algebras. This is because,
for the problems we will deal with, it is often possible to reduce to the standard graded case: If $R$ is a $\N$-graded $\kk$-algebra,
then there exists a high enough natural number $m$ such that the Veronese ring $R^{(m)}=\bigoplus_{i\in \N}R_{im}$ is standard graded.
Furthermore, 
\[\Lyu(R^{(m)})=\Lyu(R).\]
Notice also that $\depth R^{(m)}\geq \depth R$ for all $m\geq 1$ by \cite[Theorem 3.1.1]{GW}.
\end{remark}

\begin{example1}\label{ex:SR}
Let $S=\kk[X_1,\ldots ,X_n]$ be the polynomial ring in $n$ variables over a field $\kk$. 
Given a finite simplicial complex $\Delta$ on $[n]$, let
\[I_{\Delta}=\left(\prod_{i\in\sigma}X_i:\sigma \in 2^{[n]}\setminus \Delta\right)\subseteq S\]
the Stanley-Reisner ideal of $\Delta$, and $\kk[\Delta]=S/I_{\Delta}$ the Stanley-Reisner ring of $\Delta$.
Denoting by $\F(\D)$ the set of facets of $\D$ and by $x_i=\overline{X_i}\in\kk[\Delta]$, we have
\[\Min(\kk[\D])=\{(x_i:i\in[n]\setminus \sigma):\sigma\in \F(\D)\}.\]
In this case, so, $\Lyu(\kk[\D])$ is the simplicial complex with $\F(\D)$ as vertex set and such that
\[\mathcal{A}\subseteq \F(\D) \mbox{ \ is a face of $\Lyu(\kk[\D])$ \ }\iff \bigcap_{\sigma\in\mathcal{A}}\sigma\neq \emptyset.\]
So, $\Lyu(\kk[\D])$ is the nerve of $\D$ with respect to the covering given by its facets. The Borsuk's nerve lemma \cite[Theorem 10.6]{Bj} therefore implies that 
$\Lyu(\kk[\D])$ and $\D$ are homotopically equivalent. In particular,
\[\widetilde{H_i}(\Lyu(\kk[\D]);\kk)\cong \widetilde{H_i}(\D;\kk) \ \ \ \forall \ i\in\N. \]
Hence, by Hochster's formula \cite[Theorem 13.13]{MS}, we infer that, if $R$ is a Stanley-Reisner ring over $\kk$ such that $\depth R\geq k$, then
\[\widetilde{H_0}(\Lyu(R);\kk)=\widetilde{H_1}(\Lyu(R);\kk)=\ldots = \widetilde{H_{k-2}}(\Lyu(R);\kk)=0.\]

\end{example1}

It is not known whether the same vanishing of the above example holds for any $\N$-graded $\kk$-algebra.  What is known so far is:

\begin{theorem}[Hartshorne, 1962, Propositions 1.1 and 2.1 in \cite{Ha}]
If $R$ is a $\N$-graded $\kk$-algebra such that $\depth R\geq 2$, then $\widetilde{H_0}(\Lyu(R);\kk)=0$.
\end{theorem}

\begin{theorem}[Katzman, Lyubeznik, Zhang, 2015, Theorem 1.3 in \cite{KLZ}]
If $R$ is a $\N$-graded $\kk$-algebra such that $\depth R\geq 3$, then $\widetilde{H_0}(\Lyu(R);\kk)=\widetilde{H_1}(\Lyu(R);\kk)=0$.
\end{theorem}

Another interesting problem concerns the behavior of the Lyubeznik complex under hyperplane sections. For example, in the case in which $R$ is a domain, then $\Lyu(R)$ is just a point, but there can be some graded element $x\in R$ such that $\Lyu(R/xR)$ is quite complicated. Not arbitrarily complicated however, indeed:

\begin{theorem}[Bertini, $\sim 1920$, Expos\'{e} XIII, Th\'{e}or\`{e}me 2.1 in \cite{SGA2}]
If $R$ is a $\N$-graded domain of dimension at least $3$ and $x\in R_i$ for some $i>0$, then $\widetilde{H_0}(\Lyu(R/xR);\kk)=0$.
\end{theorem}

(In the previous theorem, and in the rest of the paper, by $\N$-graded domain we mean $\N$-graded $\kk$-algebra which is a domain). In view of the above discussion about the depth,
is hard to resist from asking:

\begin{question}\label{q:main}
Let $R$ be a $\N$-graded domain of dimension at least $4$ and let $x\in R$ be a graded element of positive degree. Is it true that $\widetilde{H_0}(\Lyu(R/xR);\kk)=\widetilde{H_1}(\Lyu(R/xR);\kk)=0$?
\end{question}

With the purpose of inquiring on the above question Lyubeznik, during the conference held in 2017 in Minneapolis for his 60th birthday, proposed the following problem:

\begin{problem}\label{p:lyu}
Find a $\N$-graded $4$-dimensional domain $R$ and a graded element $x\in R$ of positive degree such that 
\[\mathrm{Min}(xR)=\{\pp_1,\pp_2,\pp_3\}, \ \ \height(\pp_i+\pp_j)=2 \ \forall \ i\neq j, \ \ \height(\pp_1+\pp_2+\pp_3)=4.\]
\end{problem}

If $R$ and $x\in R$ were like in \ref{p:lyu}, then $\Lyu(R/xR)$ would be the (empty) triangle $\langle \{1,2\},\{2,3\},\{1,3\}\rangle$. In particular $\widetilde{H_1}(\Lyu(R/xR);\kk)\cong \kk$, so $R$ and $x\in R$ would provide a negative answer to Question \ref{q:main}. In this note we will give, in Example \ref{ex:main}, a negative answer to \ref{q:main}, but without solving \ref{p:lyu}. Before moving to the proof of \ref{ex:main}, however, I would like to briefly discuss on the existence of rings like in  \ref{p:lyu}. The same proof of \cite[Theorem 2.4]{BDV} gives the following:

\begin{theorem}\label{t:bmm}
Let $d,\delta$ be two natural numbers such that $d> \delta+1$, and assume $\mathrm{char}(\kk)=0$. For all $\delta$-dimensional simplicial complex $\Delta$, there exists a $d$-dimensional domain $R$ standard graded over $\kk$, and a homogeneous ideal $I\subseteq R$ of height 1 such that:
\begin{compactitem}
\item[{\rm (i)}] $\Lyu(R/I)=\Delta$;
\item[{\rm (ii)}] If $\Min(I)=\{\pp_1, \ldots, \pp_s\}$ and $\{i_1, \ldots, i_k\}\in \Delta$, then $\height(\pp_{i_1}+\ldots +\pp_{i_k})=k$.
\item[{\rm (iii)}] If $\widetilde{H_0}(\Delta;\kk)=0$, everything can be chosen so that $\depth R/I\geq 2$;
\item[{\rm (iv)}] If $\widetilde{H_0}(\Delta;\kk)=\widetilde{H_1}(\Delta;\kk)=0$, everything can be chosen so that $\depth R/I\geq 3$;
\end{compactitem}
\end{theorem}

In particular, by choosing $d=4$, $\delta=1$ and $\Delta$ the (empty) triangle, we can find a $4$-dimensional domain standard graded over $\kk$ and a height 1 ideal $I\subseteq R$ such that 
\[\mathrm{Min}(I)=\{\pp_1,\pp_2,\pp_3\}, \ \ \height(\pp_i+\pp_j)=2 \ \forall \ i\neq j, \ \ \height(\pp_1+\pp_2+\pp_3)=4.\]
If $I$ were principal up to radical we would have solved Problem \ref{p:lyu}. Let us see how the construction works  in this particular case:

\begin{example1}\label{quasi}
Consider $J=(X,Y,Z), \ H=(XYZ) \subseteq S=\kk[X,Y,Z,W]$, and 
\[R=\kk[J_3]\subseteq S.\] 
The ideal $I=H\cap R$ has three minimal primes, 
\[\pp_1=(X)\cap R, \ \pp_2=(Y)\cap R, \ \pp_3=(Z)\cap R,\] 
satisfying the properties 
\[\height(\pp_i+\pp_j)=2 \ \forall \ i\neq j, \ \ \height(\pp_1+\pp_2+\pp_3)=4.\]
However this example does not solve Problem \ref{p:lyu}, since $I$ is not principal: in fact, $I=(XYZ, \ XYZW^3, \ XYZW^6)$. (Up to radical $I$ is generated by $XYZ$ and $XYZW^6$, but one can argue, and this will also follow by the next theorem, that $I$ is not principal even up to radical).
\end{example1}

\begin{theorem}\label{t:pos}
Let $R$ be a $\N$-graded $\kk$-algebra, and assume that ${\rm char}(\kk)=0$.
If $X=\Proj R$ is connected, satisfies $(S_3)$, is Du Bois in codimension $2$, and $H^1(X,\O_X)= 0$, then $\widetilde{H_1}(\Lyu(R/xR);\kk)=0$ \
$\forall \ x\in R_i$ with $i>0$. In particular, if $X$ is a smooth variety and $H^1(X,\O_X)= 0$, then $\widetilde{H_1}(\Lyu(R/xR);\kk)=0$ \ $\forall \ x\in R_i$ with $i>0$.
\end{theorem}
\begin{proof}
First we notice that we can assume that $R$ is standard graded: in fact, take $m\in\N$ large enough so that $R^{(m)}$ is standard graded. We have that $\Proj R\cong \Proj R^{(m)}$ and $(\Proj R/xR)_{{\rm red}}\cong(\Proj R^{(m)}/x^mR^{(m)})_{{\rm red}}$. In particular
\[\Lyu(R/xR)=\Lyu(R^{(m)}/x^mR^{(m)}).\]
So, if we prove the statement in the standard graded case we are done.

Let $S=\kk[X_1,\ldots ,X_n]$ be the standard graded polynomial ring, $I\subseteq S$ the homogeneous ideal such that $R=S/I$ and $J=I+(f)$, where $f\in S$ is a representative of $x\in R$.  We will show that $\cd(S;J)\leq n-3$, and therefore conclude by \cite[Theorem 2.3]{KLZ}.

Consider the short graded exact sequence $0\rightarrow R(-i)\xrightarrow{\cdot x} R \rightarrow R/xR\rightarrow 0$. By taking the long exact sequence in local cohomology
\[\ldots \rightarrow H^2_{\mm}(R)_{-i} \rightarrow H^2_{\mm}(R)_0\rightarrow H^2_{\mm}(R/xR)_0\rightarrow H^3_{\mm}(R)_{-i}\rightarrow \ldots\]
By \cite[Proposition 5.5]{DMV}, $H^2_{\mm}(R)_{-i}=H^3_{\mm}(R)_{-i}=0$. So $H^2_{\mm}(R/xR)_0\cong H^2_{\mm}(R)_0=0$, that is $H^1(Y,\O_Y)=0$, where $Y = \Proj R/xR$. By the exponential sequence we have an injection 
\[H^1(Y;\Z)\hookrightarrow H^1(Y,\O_Y),\] 
therefore $H^1(Y;\Z)=0$. It is now enough to show that $H^{n-2}_J(S)$ is supported at the irrelevant ideal $(X_1,\ldots ,X_n)\subseteq S$, and we will conclude by \cite[Theorem 2.8]{Og}, \cite[Chapter IV, Theorem 1.1]{hartder} and the universal coefficient theorem. Let $P\in\Proj S$ containing $J$, and denote by $\pp=P/I$. If $\height(P)=n-1$, then $R_{\pp}$ has dimension at least 3, and satisfies $(S_2)$ (indeed $(S_3)$). So the punctured spectrum of $\widehat{S_P/JS_P}\cong \widehat{R_{\pp}}/x\widehat{R_{\pp}}$ is connected by \cite[Proposition 1.13]{Va1}, thus $H^{n-2}_J(S)_P=H^{n-2}_{JS_P}(S_P)=0$ by \cite[Theorem 2.9]{HL}. If $P\in \Proj S$ is a prime ideal containing $J$ of height $n-2$, in the same way the Hartshorne-Lichtembaum vanishing theorem will show that $H^{n-2}_J(S)_P$ vanishes, and if $\height(P)<n-2$ the Grothendieck vanishing theorem will do the job.
\end{proof}

\begin{remark}
In Example \ref{quasi}, $\Proj R=X$ is the blow-up of $\PP^3$ at a point; in particular it is a smooth variety with $H^1(X,\O_X)=0$, so it cannot serve to solve Problem \ref{p:lyu}. In general, the construction of Theorem \ref{t:bmm} produces a smooth rational variety $\Proj R$, so it cannot be used to answer negatively \ref{q:main}. 
\end{remark}

Theorem \ref{t:pos} gives hope that Question \ref{q:main} has a positive answer.

\section{Question \ref{q:main} has a negative answer}

Throughout this section, let us fix a polynomial ring $S=\kk[X_1,\ldots , X_n]$ and a monomial order $<$ on $S$.

\begin{proposition}\label{p:susp}
Assume that the answer to Question \ref{q:main} is positive. Then, if $\pp\subseteq S$ is a prime ideal such that $\dim S/\pp\geq 3$ and $\init_<(\pp)=I_\D$ is square-free, $ \widetilde{H_1}(\D;\kk)=0$.
\end{proposition}
\begin{proof}
Take a weight $w=(w_1,\ldots ,w_n)\in \N^n$ such that $\init_<(\pp)=\init_w(\pp)$ (see, for example, \cite[Proposition 1.11]{sturmfels}). Let $Z$ be a new indeterminate over $S$, and $P=S[Z]$ the polynomial ring with grading
$\deg(X_i)=w_i$ and $\deg(Z)=1$. The homogenization $\homo_w(\pp)\subseteq P$ of $\pp\subseteq S$ is still a prime ideal. So, $R=P/\homo_w(\pp)$ is a $\N$-graded domain of dimension at least $4$. If the answer to Question \ref{q:main} were positive, then 
\[\widetilde{H_1}(\Lyu(R/\overline{Z}R);\kk)=0.\]
But $R/\overline{Z}R\cong S/\init_<(\pp)=\kk[\D]$, and we saw in Example \ref{ex:SR} that the singular homologies of $\Lyu(\kk[\D])$ and $\D$ agree.
\end{proof}

The conclusion of Proposition \ref{p:susp} looks suspicious. It is related to the following:

\begin{question}\label{q:susp}
Is there a prime ideal $\pp\subseteq S$ such that $\init_<(\pp)$ is square-free but $S/\pp$ is not Cohen-Macaulay? 
\end{question}

At a first thought, probably your answer to the above question would be: ``Of course!''. And in fact, as we will briefly see, \ref{q:susp} has an affirmative answer; however, before we wish to outline that most 
known examples of prime ideals with square-free initial ideal are Cohen-Macaulay:

\begin{enumerate}
\item Determinantal-type ideals (\cite{St,HT,Co1,Co2,KM});
\item Orlik-Terao ideals (\cite{Pr,PS};
\item Multigraded multiplicity-free ideals (\cite{Br, CDG});
\item Prime ideals $\pp\subseteq S$ such that $\dim S/\pp\leq 2$ (\cite{KS}).
\end{enumerate} 

We will soon provide an example of a prime ideal as in \ref{q:susp} and contradicting the conclusion of \ref{p:susp} (and so its assumption). We need to know that ``Segre products and Gr\"obner deformations commute", which is a special case of the results achieved by Shibuta in \cite{Sh}. We will give a simple proof of the case we need for the convenience of the reader.

Let $A=\kk[Y_0,\ldots ,Y_a]$ and $B=\kk[Z_0,\ldots ,Z_b]$ be two standard graded polynomial rings over $\kk$. Given two homogeneous ideals $I\subseteq A$ and $J\subseteq B$, their Segre product $I\sharp J$ 
is the ideal of $P=\kk[X_{ij}:i=0,\ldots ,a, j=0,\ldots ,b]$ defined as the kernel of the $\kk$-algebra homomorphoism from $P$ to the Segre product $(A/I)\sharp (B/J)=\bigoplus_{d\in\N}(A/I)_d\otimes_{\kk} (B/J)_d$ mapping $X_{ij}$ to $\overline{Y_i}\cdot \overline{Z_j}$. 
%
%
Let $>_a$ the lexicographical monomial order on $A$ extending the linear order
$Y_0>\ldots >Y_a$ and $>_b$ the lexicographical monomial order on $B$ extending $Z_0>\ldots >Z_b$. Finally, let $>$ be the lexicographical monomial order on $P$ extending the linear order
$X_{00}>X_{01}>\ldots >X_{0b}>X_{10}>\ldots >X_{1b}>\ldots >X_{a0}>\ldots >X_{ab}$.

\begin{proposition}\label{p:segre}
With the above notations, if $\init_{<_a}(I)$ and $\init_{<_b}(J)$ are square-free, then $\init_{<}(I\sharp J)$ is square-free. More precisely, if $u_1,\ldots ,u_r\in A$ (resp. $v_1,\ldots ,v_s\in B$) are square-free monomial generators of $\init_{<_a}(I)$ (resp. of $\init_{<_b}(J)$), then the following are
square-free monomial generators of $\init_{<}(I\sharp J)$:


\begin{enumerate}
\item $u_i(X_{0i_0},\ldots ,X_{ai_a})$ such that $i=1,\ldots ,r$ and $b\geq i_0\geq \ldots \geq i_a\geq 0$;
\item $v_j(X_{j_00},\ldots ,X_{j_bb})$ such that $j=1,\ldots ,s$ and $a\geq j_0\geq \ldots \geq j_b\geq 0$;
\item $X_{ij}X_{hk}$ such that $0\leq i<h\leq a$ and $0\leq j<k\leq b$.
\end{enumerate}
\end{proposition}
\begin{proof}
Since $(A/I)\sharp (B/J)$ and $(A/\init_{<_a}(I))\sharp (B/\init_{<_b}(J))$ have the same Hilbert function, $\init_<(I\sharp J)$ and $\init_<(\init_{<_a}(I)\sharp \init_{<_b}(J))$ must coincide; so we can assume that $I=(u_1,\ldots ,u_r)$ and $J=(v_1,\ldots ,v_s)$. First of all we claim that $I\sharp J$ is generated by:
\begin{itemize}
\item[(i)] $u_i(X_{0i_0},\ldots ,X_{ai_a})$ such that $i=1,\ldots ,r$ and $b\geq i_0\geq \ldots \geq i_a\geq 0$;
\item[(ii)] $v_j(X_{j_00},\ldots ,X_{j_bb})$ such that $j=1,\ldots ,s$ and $a\geq j_0\geq \ldots \geq j_b\geq 0$;
\item[(iii)] $X_{ij}X_{hk}-X_{ik}X_{hj}$ such that $0\leq i<h\leq a$ and $0\leq j<k\leq b$.
\end{itemize}
To see this consider the homomorphisms of $\kk$-algebras $P\xrightarrow{\phi}A\sharp B$ and $A\sharp B\xrightarrow{\pi}(A/I)\sharp (B/J)$ defined by $\phi(X_{ij})=Y_iZ_j$ and $\phi(Y_iZ_j)=\overline{Y_iZ_j}$. Then, 
\[I\sharp J=\Ker(\pi\circ\phi)=\Ker(\phi)+\phi^{-1}(\Ker(\pi)).\]
It is well known that $\Ker(\phi)$ is the ideal of 2-minors of the matrix $X=(X_{ij})$, that is generated by the polynomials in (iii). Moreover, it is easy to see that a system of generators of $\Ker(\pi)$ is given by
$u_i\cdot m$ such that $i=1,\ldots ,r$ and $m\in B$ is a monomial of degree $\deg(u_i)$ and $m\cdot v_j$ such that $j=1,\ldots ,s$ and $m\in A$ is a monomial of degree $\deg(v_j)$. It is now clear that $\phi^{-1}(\Ker(\pi))$ is generated by the monomials in (i) and (ii). 

It remains to prove that the polynomials in (i),(ii) and (iii) form a Gr\"obner basis with respect to $<$. It is well known that the polynomials in (iii) alone are a Gr\"obner basis.
Therefore, it  just remains to prove that the $S$-polynomials of the polynomials in (iii) against the monomials in (i) and (ii) 
reduce to zero. This is simple to check.
\end{proof}

Finally, we are ready to negatively answer Question \ref{q:main}:

\begin{example1}\label{ex:main}
Let $A=\kk[Y_0,Y_1]$ and $B=\kk[Z_0,Z_1,Z_2]$. Furthermore, take 
\[g=Z_0Z_1Z_2+Z_1^3+Z_2^3\in B,\]
$I=\{0\}\subseteq A$ and $J=(g)\subseteq B$. Notice that $E=\Proj B/J$ is an elliptic curve of $\PP^2$
(not smooth, but irreducible), and that $\Proj P/(I\sharp J)\cong \PP^1\times E$. So $\pp=I\sharp J$ is a prime ideal.
Since $\init_{<_b}(g)=Z_0Z_1Z_2$, Proposition \ref{p:segre} implies that $\init_<(\pp)$ is generated by: 
\[X_{10}X_{11}X_{12}, \ X_{10}X_{11}X_{02}, \ X_{10}X_{01}X_{02}, \ X_{00}X_{01}X_{02}, \ \ \ \ \ X_{00}X_{11}, \ X_{00}X_{12}, \ X_{01}X_{12}.  \]

So $\init_<(\pp)=I_{\Delta}\subseteq P=\kk[X_{ij}:(i,j)\in \{0,1\}\times\{0,1,2\}]$ where $\Delta$ is the following simplicial complex on the vertices $\{0,1\}\times\{0,1,2\}$:

\bigskip

\bigskip

\centerline{\includegraphics[height=8cm]{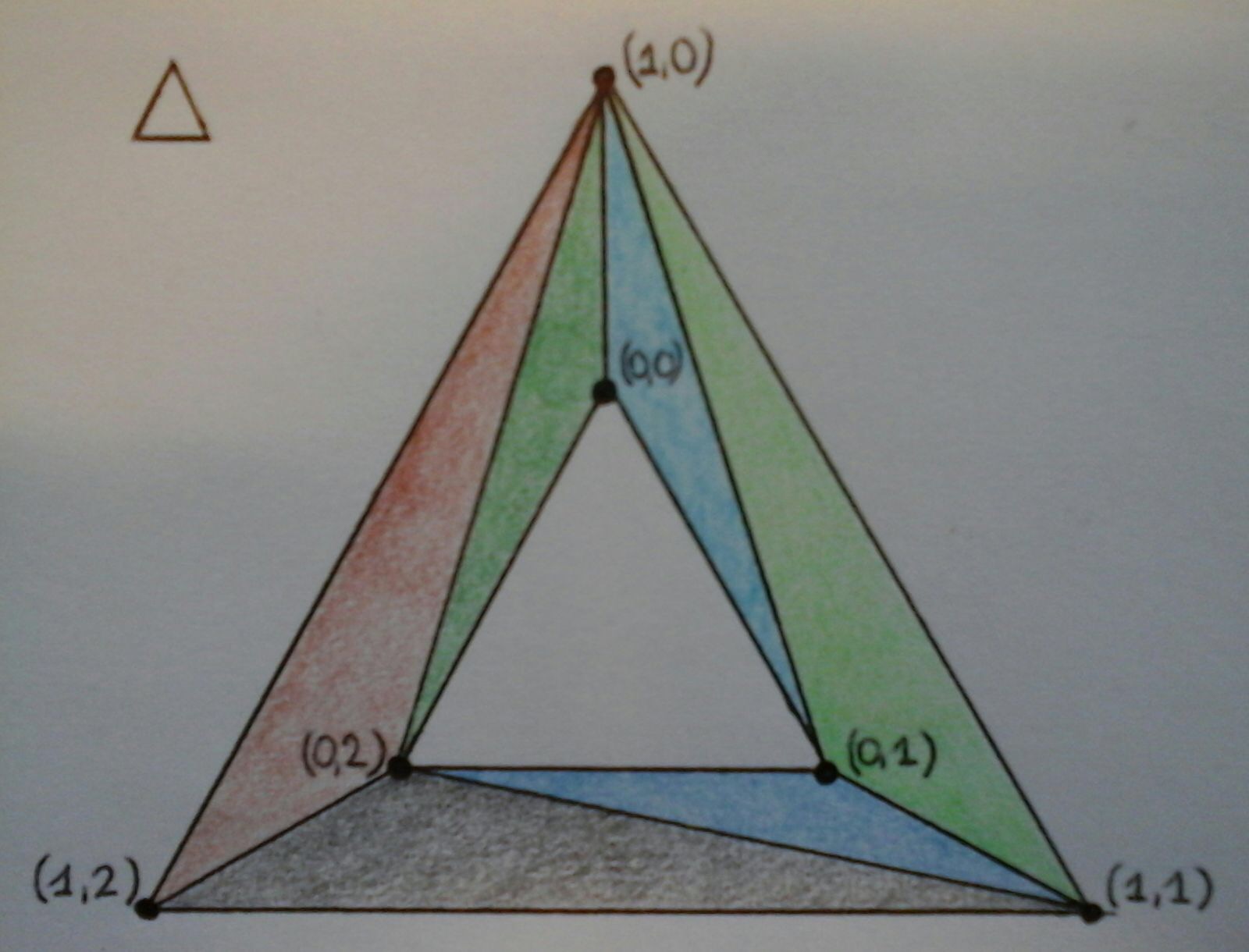}}

\bigskip

Since $\widetilde{H_1}(\Delta;\kk)\cong \kk$, this example contradicts the conclusion of Proposition \ref{p:susp}, so it provides a negative answer to Question \ref{q:main}. Notice that it also provides a positive answer to Question \ref{q:susp}: in fact $P/\pp$ is a $3$-dimensional domain and $\init_<(\pp)$ is square-free, but $P/\pp$ is not Cohen-Macaulay by \cite[Theorem 4.1.5]{GW}.

Using the notations of Proposition \ref{p:susp}, $R=P[Z]/\hom_w(\pp)$  is a $\N$-graded 4-dimensional domain and $\widetilde{H_1}(\Lyu(R/\overline{Z}R);\kk)\neq 0$.
\end{example1}

I would like to conclude the paper by proposing two problems:

\begin{problem}\label{p:final}
Given a connected simplicial complex $\Delta$ and $d> \dim \Delta+1$, find a $d$-dimensional standard graded domain $R$ and a linear form $x\in R$ such that:
\begin{compactitem}
\item[{\rm (i)}] $\Lyu(R/xR)=\Delta$;
\item[{\rm (ii)}] If $\Min(xR)=\{\pp_1, \ldots, \pp_s\}$ and $\{i_1, \ldots, i_k\}\in \Delta$, then $\height(\pp_{i_1}+\ldots +\pp_{i_k})=k$.
\end{compactitem}
\end{problem}

Notice that a positive answer to the above problem would supply a positive answer to Lyubeznik's problem \ref{p:lyu}. In order to solve it, one might try to deform the examples coming from Theorem \ref{t:bmm}.
The last problem I want to propose, that is related to Question \ref{q:susp}, came out from discussions with Alexandru Constantinescu and Emanuela De Negri, who I thank:

\begin{problem}
Find a prime ideal $\pp\subseteq S$ such that $\init_<(\pp)$ is square-free, $\Proj S/\pp$ is smooth, but $S/\pp$ is not Cohen-Macaulay. 
\end{problem}

If one found an example of a smooth Calabi-Yau projective variety with a reduced Gr\"obner degeneration, then the same construction of Example \ref{ex:main} would provide an answer to the above problem. Unfortunately I am not aware of such Calabi-Yau varieties: for sure there is no elliptic curve in $\PP^2$ like this.

%


\end{document}